\documentclass[11pt]{amsart}

\usepackage[letterpaper,margin=1in]{geometry}
\usepackage[T1]{fontenc}
\usepackage{lmodern}
\usepackage{amsmath,amssymb,amsthm}
\usepackage{microtype}
\usepackage{enumitem}
\usepackage[hidelinks]{hyperref}

\newcommand{\K}{\mathbb{K}}
\newcommand{\N}{\mathbb{N}}
\newcommand{\x}{\underline{x}}
\newcommand{\X}{\underline{X}}
\DeclareMathOperator{\rank}{rank}
\DeclareMathOperator{\End}{End}

\newtheorem{theorem}{Theorem}[section]
\newtheorem{conjecture}[theorem]{Conjecture}
\newtheorem{proposition}[theorem]{Proposition}
\setlist[enumerate]{label=\textup{(\roman*)},leftmargin=*,itemsep=3pt}

    	\author[]{Sizhuo Yan }
	\address{ State Key Laboratory of Mathematical Sciences, Academy of Mathematics and Systems Science, Chinese Academy of Sciences, Beijing 100190, China }
	\email{yansizhuo@amss.ac.cn}

    	\author[]{Jianting Yang}
	\address{ School of Mathematical Sciences, University of Electronic Science and Technology of China, Chengdu 611731,  China} 
	\email{jtyangsdp@gmail.com}

\title{On a Rank Nullstellensatz Conjecture\\
for Noncommutative Polynomials}

\date{}

\hypersetup{
  pdftitle={On a Rank Nullstellensatz Conjecture for Noncommutative Polynomials},
  pdfsubject={A counterexample to Conjecture 3 of Volcic and its homogeneous case}
}

\begin{document}
\maketitle

\begin{abstract}
Jurij Vol\v{c}i\v{c} conjectured that a noncommutative polynomial $g$
belongs to the two-sided ideal generated by $f_1,\ldots,f_l$ if and only
if, for matrices of every size, the rank of the evaluation of $g$ is bounded
by a constant times the  maximum rank of the evaluations of
$f_1,\ldots,f_l$. In this paper, we show
this equivalence when the  generators $f_1,\ldots,f_l$ are homogeneous. We then give
an explicit counterexample in the non-homogeneous case.
\end{abstract}

\section{Introduction}

In~\cite{Vol25},
Vol\v{c}i\v{c} surveys several results on noncommutative Nullstellens{\"a}tze and proposes a conjecture based on rank inequalities. In this paper, we give an affirmative answer to this conjecture when  the generators are homogeneous, and give a counterexample in the non-homogeneous case.

Throughout this paper, $\N$ denotes the set of positive integers and $\K$ denotes an algebraically closed field of characteristic 0. We use
$\x=(x_1,x_2,\ldots,x_d)$ to represent a tuple of noncommutative
variables, where $d\in\N$. Let $\K\langle\x\rangle$ denote the free
$\K$-algebra generated by variables $\x=(x_1,x_2,\ldots,x_d)$, and let $\langle\x\rangle$
denote the set of words, including the empty word $1$. The degree
$\deg(\omega)$ of a word $\omega$ is its length, and the degree
$\deg(f)$ of a nonzero polynomial $f$ is the maximum degree of a word
with nonzero coefficient. A polynomial is homogeneous if all its words
with nonzero coefficients have the same degree.

Let $\X=(X_1,X_2,\ldots,X_d)\in\mathrm{M}_n(\K)^d$ denote a tuple
of $n\times n$ matrices. The evaluation of
$f\in\K\langle\x\rangle$ at $\X$ is denoted by $f(\X)$, with the
constant $1$ evaluated as $I_n$. For $l\in\N$, we write
$(f_1,\ldots,f_l)$ for the two-sided ideal of $\K\langle\x\rangle$
generated by $f_1,\ldots,f_l$.

With the above notation, Vol\v{c}i\v{c}'s conjecture can be stated as follows.
 
\begin{conjecture}
\label{conj:rank}
\cite[Conjecture~3]{Vol25}
For $f_1,\ldots,f_l,g\in\K\langle\x\rangle$, the following
conditions are equivalent:
\begin{enumerate}
  \item there exists $C\in\N$ such that for all $n\in\N$ and
  $\X\in\mathrm{M}_n(\K)^d$,
  \begin{equation*}
    \rank g(\X)\leq C\max\{\rank f_1(\X),\ldots,\rank f_l(\X)\};
  \end{equation*}
  \item the polynomial $g$ belongs to the two-sided ideal generated by
  $f_1,\ldots,f_l$.
\end{enumerate}
\end{conjecture}

In Section~\ref{sec:homogeneous}, we prove that
Conjecture~\ref{conj:rank} holds when $f_1,\ldots,f_l$ are homogeneous.
The proof follows directly from the homogeneous Nullstellens{\"a}tze for true
zeros. In Section~\ref{sec:counterexample}, we show that the conjecture
fails in the non-homogeneous case, even for $l=1$.
A related homogeneous Nullstellensatz for joint invariant subspaces is
studied in~\cite{YYZ26}.

\section{Homogeneous Nullstellens{\"a}tze for Rank Domination}
\label{sec:homogeneous}

The true zero set of a polynomial $g\in\K\langle\x\rangle$ and the
common true zero set of $f_1,\ldots,f_l\in\K\langle\x\rangle$ are
defined by
\[
  Z(g):=\bigcup_{n\in\N}
    \{\X\in\mathrm{M}_n(\K)^d\mid g(\X)=0\},
\]
\[
  Z(f_1,\ldots,f_l):=\bigcup_{n\in\N}
    \{\X\in\mathrm{M}_n(\K)^d\mid
      f_j(\X)=0\text{ for all }j=1,\ldots,l\}.
\]

We recall the following homogeneous Nullstellensatz of Salomon, Shalit and Shamovich~\cite[Theorem~7.3]{SSS18}; see also~\cite[Theorem~2]{Vol25}.

\begin{theorem}[Salomon--Shalit--Shamovich]
\label{thm:sss}
Let $f_1,\ldots,f_l,g\in\K\langle\x\rangle$, where
$f_1,\ldots,f_l$ are homogeneous. Then the following are equivalent:
\begin{enumerate}
  \item $Z(f_1,\ldots,f_l)\subseteq Z(g)$;
  \item  $g$ belongs to the two-sided ideal generated by
  $f_1,\ldots,f_l$.
\end{enumerate}
\end{theorem}

The homogeneous case of Conjecture~\ref{conj:rank} follows from Theorem~\ref{thm:sss}. The implication $(ii)\Rightarrow(i)$ was already observed by Vol\v{c}i\v{c}~\cite{Vol25}. For completeness, we provide the details below.

\begin{theorem}
\label{thm:homogeneous}
Let $f_1,\ldots,f_l,g\in\K\langle\x\rangle$, where
$f_1,\ldots,f_l$ are homogeneous. Then the following are equivalent.
\begin{enumerate}
  \item there exists $C\in\N$ such that for all $n\in\N$ and
  $\X\in\mathrm{M}_n(\K)^d$,
  \[
    \rank g(\X)\leq C\max\{\rank f_1(\X),\ldots,\rank f_l(\X)\};
  \]
  \item the polynomial $g$ belongs to the two-sided ideal generated by
  $f_1,\ldots,f_l$.
\end{enumerate}
\end{theorem}

\begin{proof}
\textup{(ii)}$\Rightarrow$\textup{(i)}.
Suppose $g\in(f_1,\ldots,f_l)$. There exist $N\in\N$,
$p_r,q_r\in\K\langle\x\rangle$ and $j_r\in\{1,\ldots,l\}$ such that
\[
  g=\sum_{r=1}^{N}p_r f_{j_r}q_r.
\]
For every $n\in\N$ and $\X\in\mathrm{M}_n(\K)^d$, the elementary
rank inequalities give
\[
\begin{aligned}
  \rank g(\X)
  &\leq\sum_{r=1}^{N}\rank\bigl(p_r(\X)f_{j_r}(\X)q_r(\X)\bigr)\\
  &\leq N\max\{\rank f_1(\X),\ldots,\rank f_l(\X)\}.
\end{aligned}
\]
Thus \textup{(i)} holds with $C=N$. This implication does not require
homogeneity.

\textup{(i)}$\Rightarrow$\textup{(ii)}.
Let $\X\in Z(f_1,\ldots,f_l)$. Since $f_j(\X)=0$ for all
$j=1,\ldots,l$, condition~\textup{(i)} implies $\rank g(\X)=0$.
Hence $g(\X)=0$, and therefore
$Z(f_1,\ldots,f_l)\subseteq Z(g)$. By Theorem~\ref{thm:sss}, the
polynomial $g$ belongs to $(f_1,\ldots,f_l)$.
\end{proof}

Only the generators $f_1,\ldots,f_l$ are required to be homogeneous, and 
the degrees of these $f_k$ may differ, $g$ need not be homogeneous.

\section{Non-homogeneous Case}
\label{sec:counterexample}

We now show that Theorem~\ref{thm:homogeneous} cannot be extended to
non-homogeneous generators. The following counterexample gives a negative
answer to Conjecture~\ref{conj:rank} in the general case.

\begin{proposition}
\label{prop:counterexample}
Let $f_1=xy-1$ and $g=yx-1$ in $\K\langle x,y\rangle$. Then
\[
  \rank g(X,Y)=\rank f_1(X,Y)
  \qquad\text{for all } n\in\N \text{ and } X,Y\in\mathrm{M}_n(\K),
\]
but $g\notin(f_1)$. Consequently, Conjecture~\ref{conj:rank} is not true in general.
\end{proposition}

\begin{proof}
Fix $X,Y\in\mathrm{M}_n(\K)$. Suppose 
\[ v\in\ker(I-XY). \] 
Then we have $ XYv=v,$ hence 
\[ (I-YX)Yv = Yv-YXYv = Y(v-XYv) = 0, \] 
therefore
\[ Yv\in\ker(I-YX). \] 
Moreover, if $Yv=0$, then 
\[ v=XYv=X(Yv)=0. \]
 Thus, $Y$ defines an injective map from $\ker(I-XY)$ to $\ker(I-YX)$, and we obtain that 
\[ \dim\ker(I-XY) \leq \dim\ker(I-YX). \]
Similarly $ \dim\ker(I-XY) \geq \dim\ker(I-YX)$, this shows that $ \rank g(X,Y)=\rank f_1(X,Y)$.

To prove $g\notin(f_1)$, let $V$ be an infinite-dimensional $\K$-vector space with basis $\{e_k\}_{k\in\N}$.

Suppose  to the contrary that $g\in(f_1)$. Then
\[
  g=\sum_{k=1}^{N}a_k f_1 b_k
\]
for some $N\in\N$ and $a_k,b_k\in\K\langle x,y\rangle$,
 and consider the left and right shift operators
$L,R\in\End_{\K}(V)$ given by
\[
  Le_1=0,\qquad Le_{k+1}=e_k,\qquad Re_k=e_{k+1}
  \quad(k\in\N).
\]
Since $LR=\operatorname{Id}_V$, evaluating the above identity at $(L,R)$ gives
\[
  RL-\operatorname{Id}_V =g(L,R)
  =\sum_{k=1}^{N}a_k(L,R)(LR-\operatorname{Id}_V)b_k(L,R)=0.
\]
This contradicts $(RL-\operatorname{Id}_V)e_1=-e_1\neq 0$.
Therefore $g\notin(f_1)$.
Thus condition~\textup{(i)} of Conjecture~\ref{conj:rank} holds with
$C=1$, while condition~\textup{(ii)} fails.

\end{proof}

  \section*{Acknowledgments}

The authors thank Junqi Liu for his assistance with the use of AI tools during the preparation of this manuscript.

\section*{Declaration of generative AI}

The main results presented in this paper are obtained through an interaction between the authors and an artificial intelligence agent system, MechMath Agent Team (MMAT) \cite{cao2026mechmathagentteamllm}.  GPT-5.6 Sol was used during this work to assist in exploring proof strategies and revising the exposition. The authors assume full responsibility for all proofs and mathematical details and for the mathematical correctness of the final results.
\bibliographystyle{plain}
\bibliography{rank_nullstellensatz_references}
\end{document}